\documentclass[12pt]{article}
\pdfoutput=1
\usepackage[margin=1.25in]{geometry}
\usepackage{amscd}
\usepackage{amssymb,amsmath}
\usepackage{amsthm}
\usepackage{hyperref,xcolor}
\usepackage{graphicx}
\usepackage{cases}
\usepackage{indentfirst}
\usepackage{mathrsfs}
\usepackage{textcomp}
\usepackage{enumerate}
\usepackage[marginal]{footmisc}
\hypersetup{
	colorlinks,%
	citecolor=blue,%
	filecolor=blue,%
	linkcolor=blue,%
	urlcolor=black
}

\usepackage{verbatim}
\usepackage{sectsty}

\makeatletter

\newtheorem{theorem}{Theorem}[section]

\newtheorem{lemma}[theorem]{Lemma}
\newtheorem{remark}[theorem]{Remark}

\newtheorem{corollary}[theorem]{Corollary}

\numberwithin{equation}{section}

\newdimen\bibspace
\renewenvironment{thebibliography}[1]{%
	\section*{\refname %or \bibname if you use ``book'' as the documentclass
		\@mkboth{\MakeUppercase\refname}{\MakeUppercase\refname}}%
	\list{\@biblabel{\@arabic\c@enumiv}}%
	{\settowidth\labelwidth{\@biblabel{#1}}%
		\leftmargin\labelwidth
		\advance\leftmargin\labelsep
		\itemsep\bibspace
		\parsep\z@skip     %
		\@openbib@code
		\usecounter{enumiv}%
		\let\p@enumiv\@empty
		\renewcommand\theenumiv{\@arabic\c@enumiv}}%
	\sloppy\clubpenalty4000\widowpenalty4000%
	\sfcode`\.\@m}
{\def\@noitemerr
	{\@latex@warning{Empty `thebibliography' environment}}%
	\endlist}

\makeatother
\makeatletter

\def\XXint#1#2#3{{\setbox0=\hbox{$#1{#2#3}{\int}$}
		\vcenter{\hbox{$#2#3$}}\kern-.5\wd0}}

\newcommand{\be}{\begin{equation}}      \newcommand{\ee}{\end{equation}}

\begin{document}	
	\title{\bf\Large Entire solutions of the generalized Hessian inequality
		\footnotetext{\hspace{-0.35cm}Jiguang Bao
			\endgraf jgbao@bnu.edu.cn
			\vspace{0.25cm}
			\endgraf Xiang Li
			\endgraf 202031130022@mail.bnu.edu.cn
			\vspace{0.25cm}
			\endgraf Jing Hao
			\endgraf 202021130025@mail.bnu.edu.cn
						\vspace{0.25cm}
			\endgraf All authors are supported in part by the National Natural Science Foundation of China (11871102).
	}}
	\vspace{0.25cm}
	\author{Xiang Li,\ \ Jing Hao,\ \ Jiguang Bao\footnote{School of Mathematical Sciences, Beijing Normal University,
			Laboratory of Mathematics and Complex Systems, Ministry of Education, Beijing 100875, China}}
	\date{}
	\maketitle
\vspace{-0.8cm}
\noindent{\bf Abstract}		
 In this paper, we discuss the more  general Hessian inequality $\sigma_{k}^{\frac{1}{k}}(\lambda (D_i (A$ $\left(|Du|\right) D_j u)))\geq f(u)$ including the Laplacian, p-Laplacian, mean curvature, Hessian, k-mean curvature operators, and provide a necessary and sufficient condition on the global solvability, which can be regarded as  generalized Keller-Osserman conditions.
	

	\section{Introduction and the statement of results}
In this paper, we discuss the solvability of the generalized Hessian  inequality
\begin{equation}\label{101}
\sigma_{k}^{\frac{1}{k}}\left(\lambda\left(D_i\left(A\left(|Du|\right)D_j u\right)\right)\right) \geq f(u)  \text { in }  \mathbb{R}^{n},
\end{equation}
where
 $$
 \sigma_{k}(\lambda)=\sum_{1 \leq i_{1}<\cdots<i_{k} \leq n} \lambda_{i_{1}} \cdots \lambda_{i_{k}},  \quad \lambda=\left(\lambda_{1}, \lambda_{2}, \cdots, \lambda_{n}\right) \in \mathbb{R}^{n}, \ k=1,2, \cdots, n
 $$
 is the k-th elementary symmetric function, $\lambda\left(D_i\left(A\left(|Du|\right)D_j u\right)\right)$ denotes the eigenvalues  of the symmetric matrix of $\left(D_i\left(A\left(|Du|\right)D_j u\right)\right)$, and $A$, $f$ are two given positive continuous functions on $(0,+\infty)$.
 
 The generalized Hessian operator $\sigma_{k}\left(\lambda\left(D_i\left(A\left(|Du|\right)D_j u\right)\right)\right)$, introduced by   many authors \cite{r2,r3,r4}, is an important class of fully nonlinear operator. It is a generalization of some typical operators we shall be interested in  as follows:   the m-k-Hessian operator for the case $ A(p)=p^{m-2}$, $m>1$ is treated by Trudinger and Wang \cite{r5};
  the k-mean curvature operator for the case $ A(p)=\left(1+p^{2}\right)^{-\frac{1}{2}}$  is treaded by  Concus and Finn \cite{r19} and Peletier and  Serrin \cite{r20}; 
  the generalized k-mean curvature operator for the case $A(p)=\left(1+p^{2}\right)^{-\alpha}$, $\alpha < \frac{1}{2}$ and $ A(p)=p^{2 m-2}\left(1+p^{2 m}\right)^{-\frac{1}{2}}$, $m>1$
  is treated  by  Tolksdorf \cite{r7},  Usami \cite{r8} and Suzuki  \cite{r9}, respectively.

 In particular, (\ref{101}) is the k-Hessian innequality for the case $ A(p)=1$.  For $k=1$, Wittich (n=2 \cite{r24}),  Haviland (n=3 \cite{r25}), Walter ($n\geq2$ \cite{r26}) proved the Laplacian equation
  \begin{equation*}
 \Delta u=f(u) \text { in } \mathbb{R}^{n}
 \end{equation*}
 has no solution  if  and only if
  \begin{equation*}
 \int^{\infty}\left(\int^{s} f(t) d t\right)^{-\frac{1}{2}} d s<\infty.
 \end{equation*}
 Here and after, we omit the lower limit to admit an arbitrary positive number. Keller \cite{r15} and Osserman \cite{r16} showed that the Laplacian inequality 
   \begin{equation*}
 \Delta u\geq f(u) \text { in } \mathbb{R}^{n}
 \end{equation*}
has a positive  solution $u\in C^2\left(\mathbb{R}^{n}\right)$ if and only if $f$ satisfies the Keller-Osserman condition
 \begin{equation}\label{111}
 \int^{\infty}\left(\int^{s} f(t) d t\right)^{-\frac{1}{2}} d s=\infty.
 \end{equation}
 The condition (\ref{111}) is often used to study the  boundary blow-up (explosive, large) solutions (see \cite{r21,r22,r23}). Ji and Bao \cite{r13} extended the above results from $k=1$ to 
 $1\leq k\leq n$,  which can be regardes as the generalized Keller-Osserman condition. Naito and Usami \cite{r10} extended the above results  from $A(p)=1$ to the generalized Hessian inequality  (\ref{101}) for $k=1$ and got similar results.

 In this paper, we shall extend this result  from $k=1$ to 
 $1\leq k\leq n$ for the generaralized Hessian inequality (\ref{101}) and develop existence and nonexistence conditions of entire solutions  for (\ref{101}). 
  To state our results, we define a generalized k-convex  entire solution of (\ref{101})  to be a function $u\in\Phi^{k}\left(\mathbb{R}^{n}\right)$ which satisfies (\ref{101}) at each $x \in \mathbb{R}^{n}$, where
 $$
\Phi^{k}\left(\mathbb{R}^{n}\right)=\left\{u \in  C^{1}\left(\mathbb{R}^{n}\right): A\left(|Du|\right)D u\in C^{1}\left(\mathbb{R}^{n}\right),\  \lambda\left(D_i\left(A\left(|Du|\right)D_j u\right)\right) \in \Gamma_{k} \text { in } \mathbb{R}^{n}\right\},
 $$
and
 $$
 \Gamma_{k}:=\left\{\lambda \in \mathbb{R}^{n}: \sigma_{l}(\lambda)>0,\ l=1,2, \cdots, k\right\}.
 $$

 In $(\ref{101})$, we assume that the positive function $A\in C^1(0, \infty)$ satisfies
  \begin{equation}\label{103}
  p A(p) \in C[0, \infty) \text { is strictly monotone increasing  in }(0,\infty),
 \end{equation}
  and the positive function $f\in C(0,\infty)$  satisfies 
 \begin{equation}\label{109}
 f \text { is  monotone non-decreasing  in }(0,\infty).
 \end{equation}

 First we consider the case 
 \begin{equation}\label{114}
 	\lim _{p \rightarrow \infty} p A(p)<\infty.
 \end{equation}

 \begin{theorem}\label{102}
 	Assume that $A$  satisfies (\ref{103}), (\ref{114}) and $f$ satisfies (\ref{109}), then the inequality (\ref{101}) has no positive solution  $u\in C^{2}\left(\mathbb{R}^{n} \backslash\{0\}\right) \cap \Phi^{k}\left(\mathbb{R}^{n}\right)$.
 \end{theorem}
 \begin{remark}
	The k-mean curvature innequality (\ref{101}) for the case $ A(p)=\left(1+p^{2}\right)^{-\frac{1}{2}}$ satisfies the Theorem \ref{102}, and the corresponding results were obtained by Cheng and Yau \cite{r11} and Tkachev  \cite{r12}.
\end{remark}
Next we consider the case 
\begin{equation}\label{116}
	 \lim _{p \rightarrow \infty} p A(p)=\infty.
\end{equation}
 We now define a continuous function $\Psi:[0, \infty) \rightarrow[0, \infty)$ by
  \begin{equation}\label{107}
  	 \Psi(p)=p \left(p A\left(p\right)\right)^k-\int_{0}^{p} \left(t A(t)\right)^k d t, \ p \geq 0.
  \end{equation}
  It follows from the condition (\ref{103}) that the inverse function of $\Psi$ on $[0,\infty)$ exists, denoted by $\Psi^{-1}$. For example, if $A(p)=p^{m-2},\ m>1$,    then 
  $$
  \Psi(p)=\frac{m-1}{m} p^{m} \ \text { and } \ \Psi^{-1}(p)=\left(\frac{m}{m-1}p\right)^{\frac{1}{m}}.
  $$
  A necessary and sufficient result is as follows.
 
 \begin{theorem}\label{115}
 	Assume that $A$  satisfies (\ref{103}), (\ref{116}) and $f$ satisfies (\ref{109}),
 	then inequality (\ref{101}) has  a positive solution $u\in C^{2}\left(\mathbb{R}^{n} \backslash\{0\}\right) \cap \Phi^{k}\left(\mathbb{R}^{n}\right)$
 	if and only if
 	\begin{equation}\label{110}
\int^{\infty}\left(\Psi^{-1}\left(\int^{s} f^k(t) d t\right)\right)^{-1} d s=\infty.
\end{equation}
 \end{theorem}
For $k=1, A(p)=1$, (\ref{110}) is exactly  the  Keller-Osserman condition (\ref{111}).  So we can regard (\ref{110})  as a generalized Keller-Osserman condition.

If we strengthen the case (\ref{116}) to
 \begin{equation}\label{108}
 	 0<\liminf _{p \rightarrow \infty} \frac{ A(p)}{p^{m-2}} \leq \limsup _{p \rightarrow \infty} \frac{ A(p)}{p^{m-2}}<\infty \ \text { for some } \ m>1.
 \end{equation}
 
 As a consequence of Theorems \ref{115}, we obtain the following corollary.
 \begin{corollary}\label{113}
 	Assume that $A$  satisfies (\ref{103}),  (\ref{108}) and $f$ satisfies (\ref{109}), then inequality (\ref{101}) has a positive  solution $u\in C^{2}\left(\mathbb{R}^{n} \backslash\{0\}\right) \cap \Phi^{k}\left(\mathbb{R}^{n}\right)$ if and only if
 	\begin{equation}\label{112}
 	\int^{\infty}\left(\int^{s} f^k(t) d t\right)^{-\frac{1}{k(m-1)+1}} d s=\infty.
 	\end{equation}
 \end{corollary}

	\begin{remark}
		 Corollary \ref{113} holds for the cases $A(p)=1$, $m=2$ which was obtained by Ji and Bao \cite{r13};
		$A(p)=p^{m-2}$,  $m>1$ which was obtained by Feng and Bao \cite{r14};
		$A(p)=\left(1+p^{2}\right)^{-\alpha}$, $ m=2-2\alpha>1$ and
		 $A(p)=p^{2 m-2}\left(1+p^{2 m}\right)^{-\frac{1}{2}}$, $m > 1$, which are first obtained by  authors of this paper.

	\end{remark}

	\begin{remark}
		Under the assumption of Corollary \ref{113}, if $f(u)=u^{\gamma}$, $\gamma\geq0$, then the inequality (\ref{101})
		has a positive solution $ u\in C^{2}\left(\mathbb{R}^{n} \backslash\{0\}\right) \cap \Phi^{k}\left(\mathbb{R}^{n}\right)$ if and only if $\gamma \leq m-1$.
	\end{remark}

If we strengthen the requirement of $f$ from (\ref{109}) to the positive function $f\in C(\mathbb{R})$  satisfying 
\begin{equation}\label{117}
f \text { is  monotone non-decreasing  in } \mathbb{R} \text {, }
\end{equation}
then we have similary following corollary:
\begin{corollary}\label{118}
		Assume that $A$  satisfies (\ref{103}) and $f$ satisfies (\ref{117}).
		 If (\ref{114}) holds, then the inequality (\ref{101}) has no  solution  $u\in C^{2}\left(\mathbb{R}^{n} \backslash\{0\}\right) \cap \Phi^{k}\left(\mathbb{R}^{n}\right)$;
		if (\ref{116}) holds,
		 then the inequality (\ref{101}) has a  solution $u \in C^{2}\left(\mathbb{R}^{n} \backslash\{0\}\right) \cap \Phi^{k}\left(\mathbb{R}^{n}\right)$ if and only if (\ref{110}) holds,
	in particular, if  (\ref{108}) holds, then the inequality (\ref{101}) has a  solution $u \in C^{2}\left(\mathbb{R}^{n} \backslash\{0\}\right) \cap \Phi^{k}\left(\mathbb{R}^{n}\right)$ if and only if (\ref{112}) holds.
\end{corollary}
\begin{remark}
	 Under the assumption of Corollary \ref{118}, if $f(u)=e^{u}$, then the inequality (\ref{101}) has no solution $u\in C^{2}\left(\mathbb{R}^{n} \backslash\{0\}\right) \cap \Phi^{k}\left(\mathbb{R}^{n}\right)$.
\end{remark}
 In particular, we will get  a better regularity of solutions later about that if 
$A\in C^1[0,\infty), A(0)\neq 0$, then $u\in C^{2}\left(\mathbb{R}^{n}\right) \cap \Phi^{k}\left(\mathbb{R}^{n}\right)$;
otherwise we consider the case 
\begin{equation}\label{211}
0<\liminf _{p \rightarrow 0} \frac{A(p)}{p^{l-2}} \leq \limsup _{p \rightarrow 0} \frac{ A(p)}{p^{l-2}}<\infty \ \text { for some } \ l> 2,
\end{equation}
then $u\in W_{loc}^{2, n q}\left(\mathbb{R}^{n}\right),1<q<\frac{l-1}{l-2}$, by embedding theorem, we have $u\in C^{1, \alpha}\left(\mathbb{R}^{n}\right) \cap \Phi^{k}\left(\mathbb{R}^{n}\right)$ for some $\alpha \in(0,1)$. 

The rest of our paper is organized as follows. In Section 2, we introduce some results on radial solutions and the local existence of Cauchy problem  associated to (\ref{101}) as preliminaries. In Section 3, we  give the comparison principle and prove Theorems $\ref{102}$, $\ref{115}$ and Corollary $\ref{113}$, \ref{118}.
\section{Preliminary results on radial solutions}
We need some properties of radial functions in the proof of the main theorem. For $R>0$, let $B_{R}:=\left\{x \in \mathbb{R}^{n}:|x|<R\right\}$.
	\begin{lemma}\label{206}
	For any positive number a, assume $\varphi(r) \in C[0, R) \cap C^{1}(0, R)$ is a positive solution of the Cauchy problem of the implicit equation
	\begin{equation}\label{204}
	\left\{\begin{aligned}
		&A\left(|\varphi^{\prime}(r)|\right)\varphi^{\prime}(r)
	=\left(\frac{nr^{k-n}}{C_{n}^k} \int_{0}^{r} s^{n-1} f^{k}(\varphi(s)) d s\right)^{\frac{1}{k}}=:F(r,\varphi), \ r>0 , \\
	&\varphi(0) =a.
	\end{aligned}\right.
	\end{equation}
	Then $\varphi^{\prime}(0)=0$, $\varphi^{\prime}(r)>0$ in $(0, R)$, and it satisfies $\varphi(r)\in C^{1}[0, R) \cap C^{2}(0, R) $  with $A\left(\varphi^{\prime}\left(r\right)\right)\varphi^{\prime}\left(r\right)\in C^{1}[0, R)$,  and the ordinary differential equation
	\begin{equation}\label{203}
	\begin{aligned}
	&C_{n-1}^{k-1} \left(A\left(\varphi^{\prime}(r)\right)\varphi^{\prime}(r)\right)'\left(\frac{A\left(\varphi^{\prime}(r)\right)\varphi^{\prime}(r)}{r}\right)^{k-1}
	+C_{n-1}^k\left(\frac{A\left(\varphi^{\prime}(r)\right)\varphi^{\prime}(r)}{r}\right)^{k}\\
	&=\frac{C_{n}^{k}r^{1-n}}{n}\left(r^{n-k}\left(A\left(\varphi^{\prime}(r)\right)\varphi^{\prime}(r)\right)^k\right)'=f^{k}(\varphi(r)).
	\end{aligned}
	\end{equation}
\end{lemma}
\begin{proof}
	Denote
	$$h(r):=\int_{0}^{r} A\left(|\varphi^{\prime}(s)|\right)\varphi^{\prime}(s) ds,$$
	then it satisfies $h(0)=0$ and
	$$h'(r)=A\left(|\varphi^{\prime}(r)|\right)\varphi^{\prime}(r)=\left(\frac{nr^{k-n}}{C_{n}^k} \int_{0}^{r} s^{n-1} f^{k}(\varphi(s)) d s\right)^{\frac{1}{k}}>0,\ 0<r<R.$$
	 It is easy to see that $h(r) \in C^{2}(0, R)$. By (\ref{103}) and (\ref{204}), $\varphi^{\prime}(r)>0$ in $(0, R)$.

	$$
	\lim _{r \rightarrow 0} \frac{h(r)-h(0)}{r-0}=\lim _{r \rightarrow 0} h^{\prime}(\xi)=\lim _{\xi \rightarrow 0}\left(\frac{n\xi^{k-n}}{C_{n}^k} \int_{0}^{\xi} s^{n-1} f^{k}(\varphi(s)) d s\right)^{\frac{1}{k}}=0,
	$$	
	where $\xi=\xi(r) \in(0, r)$. Hence, $h^{\prime}(0)=0$ and $h(r) \in C^1[0, R)$, which implies that $\varphi^{\prime}(0)=0$ and  $\varphi(r) \in C^{1}[0, R)$. One can see that
	\begin{equation}\label{210}
	\begin{aligned}
	\lim _{r \rightarrow 0} \frac{h^{\prime}(r)-h^{\prime}(0)}{r-0}=\lim _{r \rightarrow 0}\left(\frac{n\int_{0}^{r} s^{n-1} f^{k}(\varphi(s)) d s}{C_{n}^k r^{n}}\right)^{\frac{1}{k}}
	=\left(\frac{f^{k}(a)}{C_n^k}\right)^{\frac{1}{k}},
	\end{aligned}
	\end{equation}
	consequently, $h(r) \in C^{2}[0, R)$, which implies that $A\left(\varphi^{\prime}\left(r\right)\right)\varphi^{\prime}\left(r\right)\in C^{1}[0, R)$.
 
 A direct calculation using (\ref{204}) leads to
 \begin{equation}\label{220}
 \begin{aligned}
 h^{\prime \prime}(r) &=\frac{(h^{\prime}(r))^{1-k}}{k}\left(\frac{n(k-n) r^{k-n-1}}{C_n^k} \int_{0}^{r} s^{n-1} f^{k}(\varphi(s)) d s+\frac{n r^{k-1}}{C_n^k} f^{k}(\varphi(r))\right) \\
 &\geq \frac{1}{C_n^k} \left(\frac{h'(r)}{r}\right)^{1-k}f^k(\varphi\left(r\right))>0,
 \end{aligned}
 \end{equation}
 then $h'(r) \in C^{1}[0, R)$ is a strictly monotone increasing  function of $r$, and by (\ref{103}), 	$g\left(\varphi^{\prime}\right):=A\left(\varphi^{\prime}\right)\varphi^{\prime} \in C^{1}(0, \varphi^{\prime}\left(R\right))$ is a strictly monotone increasing  function of $\varphi^{\prime}$, then there exists inverse function $\varphi^{\prime}\left(r\right)=g^{-1}\left(h'\left(r\right)\right)\in C^{1}(0, R)$, which implies $\varphi\left(r\right)\in C^{2}(0, R)$.
	
By  $(\ref{220})$ and (\ref{204}), we have
$$ h^{\prime \prime}(r)=\frac{k-n}{k} \frac{h^{\prime}(r)}{r}+\frac{n}{ kC_n^k}\left(\frac{h^{\prime}(r)}{r}\right)^{1-k} f^{k}(\varphi(r)),$$
it is easy to verify that $\varphi(r)$ satisfies equation $(\ref{203})$.
\end{proof}

\begin{remark}\label{214}
	In particular, if $A\in C^1[0,R), A(0)\neq 0
	$, consider the function $H(r,\varphi^{\prime})=h'(r)-g(\varphi^{\prime})=0$, then $H_{\varphi^{\prime}}(0,0)=A(0)\neq 0$, hence
	 we know from the implicit function theorem that
	 there exists $\varphi^{\prime}\left(r\right)\in C^{1}[0, R)$, then we can strengthen  the regularity to $\varphi\left(r\right)\in C^{2}[0, R)$.
	Otherwise we assume that (\ref{211}) and
	by (\ref{210}), we have
	\begin{equation}\label{213}
		\left(\frac{f^{k}(a)}{C_n^k}\right)^{\frac{1}{k}}=\lim _{r \rightarrow 0} \frac{h^{\prime}(r)}{r}
		=\lim _{r \rightarrow 0} \frac{\left(\varphi'\left(r\right)\right)^{l-1}}{r}
		=\lim _{r \rightarrow 0} \frac{\varphi''\left(r\right)}{r^{-\frac{l-2}{l-1}}},
	\end{equation}
	for $1<q<\frac{l-1}{l-2}$, we can  strengthen  the regularity to $\varphi\left(r\right)\in C^{2}(0, R) \cap W^{2,q}\left(0,R\right)$.

\end{remark}

	\begin{lemma}\label{205}
		For any positive number a, assume $\varphi(r) \in C[0, R) \cap C^{1}(0, R)$ is a positive solution of the Cauchy problem (\ref{204}).
		Then, for $u(x)=\varphi(r)$, where $r=|x|<R$,  we have that 
			\begin{equation}\label{201}
		\begin{aligned}
		\lambda&\left(D_i\left(A\left(|Du|\right)D_j u\right)\right)\\
		&=
		\left(\left(A\left(\varphi^{\prime}(r)\right)\varphi^{\prime}(r)\right)', \frac{A\left(\varphi^{\prime}(r)\right)\varphi^{\prime}(r)}{r}, \cdots, \frac{A\left(\varphi^{\prime}(r)\right)\varphi^{\prime}(r)}{r}\right)
		,\ r\in [0,R),
		\end{aligned}
		\end{equation}
			and then
	 $u(x) \in C^{2}\left(B_{R} \backslash\{0\}\right) \cap \Phi^{k}\left(B_{R}\right)$ is a solution of 
	 \begin{equation}\label{221}
	 \sigma_{k}\left(\lambda\left(D_i\left(A\left(|Du|\right)D_j u\right)\right)\right)=\frac{C_{n}^{k}r^{1-n}}{n}\left(r^{n-k}\left(A\left(\varphi^{\prime}(r)\right)\varphi^{\prime}(r)\right)^k\right)'=f^k(u).
	 \end{equation}
	\end{lemma}

\begin{proof}
	By Lemma \ref{206}, we can get  $\varphi(r)\in C^{1}[0, R) \cap C^{2}(0, R) $  with $A\left(\varphi^{\prime}\left(r\right)\right)\varphi^{\prime}\left(r\right)\in C^{1}[0, R)$, and it satisfies  $\varphi^{\prime}(0)=0$, $\varphi^{\prime}(r)>0$ in $(0, R)$. For $u(x)=\varphi(r)$, where $r=|x|\in(0, R),\ i, j=1, \cdots, n $,  we have 
	\begin{equation}\label{208}
			u_i(x)=\varphi^{\prime}(r) \frac{x_i}{r},
	\end{equation}
	\begin{equation}\label{209}
		|D u|=\left|\varphi^{\prime}(r) \frac{x}{r}\right|=\varphi^{\prime}(r),
	\end{equation}
		$$
	u_{ij}(x)=\left(\frac{\varphi^{\prime \prime}(r)}{r^{2}}\right) x_{i} x_{j}-\left(\frac{\varphi^{\prime}(r)}{r^{3}}\right) x_{i} x_{j}+\left(\frac{\varphi^{\prime}(r)}{r}\right) \delta_{i j}.
	$$
	Then by (\ref{208}) and (\ref{209}), we have
	\begin{equation}\label{207}
	\begin{aligned}
	&D_i\left(A\left(|Du|\right)  D_j u\right)
	=D_i\left(A\left(\varphi^{\prime}(r)\right)\varphi^{\prime}(r)\frac{x_j}{r}\right) \\
	&	= \left(A\left(\varphi^{\prime}(r)\right)\varphi^{\prime}(r)\right)'\frac{x_ix_j}{r^2}+A\left(\varphi^{\prime}(r)\right)\varphi^{\prime}(r)\frac{\delta_{i j}r-x_j\frac{x_i}{r}}{r^2}  \\
	&=\left(\left(A\left(\varphi^{\prime}(r)\right)\varphi^{\prime}(r)\right)'-\frac{A\left(\varphi^{\prime}(r)\right)\varphi^{\prime}(r)}{r}\right)\frac{x_ix_j}{r^2}+\frac{A\left(\varphi^{\prime}(r)\right)\varphi^{\prime}(r)}{r}\delta_{i j}.
	\end{aligned}
	\end{equation}
	By (\ref{208}) and $\varphi^{\prime}(0)=0$, we have
	$$
	0
	\leq\lim _{x \rightarrow 0} |u_i(x)|
	=\lim _{x \rightarrow 0} |\varphi^{\prime}(r)||\frac{x_i}{r}|
	\leq\lim _{r \rightarrow 0}\varphi^{\prime}(r)
	=0,
	$$
	which means
	$$
	\lim _{x \rightarrow 0} u_i(x)=0.
	$$
	Similarly, using (\ref{207}) we have
	$$
	\begin{aligned}
	&\lim _{x \rightarrow 0} D_i\left(A\left(|Du|\right)D_j u\right)\\
	&=\lim _{x \rightarrow 0}
	\left(\left(\left(A\left(\varphi^{\prime}(r)\right)\varphi^{\prime}(r)\right)'-\frac{A\left(\varphi^{\prime}(r)\right)\varphi^{\prime}(r)}{r}\right)\frac{x_ix_j}{r^2}+\frac{A\left(\varphi^{\prime}(r)\right)\varphi^{\prime}(r)}{r}\delta_{i j}\right)\\
	&=\left(A\left(\varphi^{\prime}(0)\right)\varphi^{\prime}(0)\right)' \delta_{i j}.
	\end{aligned}
	$$
	Define
	$$
	u_i(0)=0, D_i\left(A\left(|Du|\right)D_j u\right)\left(0\right)=\left(A\left(\varphi^{\prime}(0)\right)\varphi^{\prime}(0)\right)' \delta_{i j}.
	$$
	then $u(x) \in C^{1}\left(B_{R}\right) \cap C^{2}\left(B_{R} \backslash\{0\}\right)$, with $A\left(|Du|\right)D u \in C^{1}\left(B_{R}\right) $.
	
    For $ r \in[0, R)$,	let
	\begin{equation*}
	a=\left\{\begin{array}{l}
	\frac{\left(A\left(\varphi^{\prime}(r)\right)\varphi^{\prime}(r)\right)'}{r^2}-\frac{A\left(\varphi^{\prime}(r)\right)\varphi^{\prime}(r)}{r^3},\ r \in(0, R), \\
	0,\ r=0,
	\end{array}\right.
	b=\left\{\begin{array}{l}
	\frac{A\left(\varphi^{\prime}(r)\right)\varphi^{\prime}(r)}{r},\ r \in(0, R), \\
	\left(A\left(\varphi^{\prime}(0)\right)\varphi^{\prime}(0)\right)',\ r=0,
	\end{array}\right.
	\end{equation*}
	then the matrix
	$$
	D_i\left(A\left(|Du|\right)D_j u\right)=a x^{T} x+b I.
	$$	
	By calculations of linear algebra, we know that the eigenvalues of a symmetric matrix such as $D_i\left(A\left(|Du|\right)D_j u\right)$ is $\left(a r^{2}+b, b, \cdots, b\right)$. Hence
	\begin{equation*}
	\lambda\left(D_i\left(A\left(|Du|\right)D_j u\right)\right)
	=\left\{\begin{array}{l}
	\left(\left(A\left(\varphi^{\prime}(r)\right)\varphi^{\prime}(r)\right)', \frac{A\left(\varphi^{\prime}(r)\right)\varphi^{\prime}(r)}{r}, \cdots, \frac{A\left(\varphi^{\prime}(r)\right)\varphi^{\prime}(r)}{r}\right), r \in(0, R), \\
	\left(\left(A\left(\varphi^{\prime}(0)\right)\varphi^{\prime}(0)\right)',\cdots, \left(A\left(\varphi^{\prime}(0)\right)\varphi^{\prime}(0)\right)'\right), r=0.
	\end{array}\right.
	\end{equation*}
	Since
	$$
	\lim _{r \rightarrow 0} \frac{A\left(\varphi^{\prime}(r)\right)\varphi^{\prime}(r)}{r}=	\left(A\left(\varphi^{\prime}(0)\right)\varphi^{\prime}(0)\right)',
	$$
	we can always think that (\ref{201}) holds, 
	 and 
	 equation (\ref{221}) can thus be obtained easily from the definition of $\sigma_{k}$.
	
		Since $f$ and $\varphi$ are both monotone non-decreasing, for $r \in [0, R)$,
	$$
	f(\varphi(r)) \geq f(\varphi(0))=f(a)>0.
	$$
	Then  we know that $\frac{A\left(\varphi^{\prime}(r)\right)\varphi^{\prime}(r)}{r}>0$ and
	$$
	\begin{aligned}
	&\sigma_{k}\left(\lambda\left(D_i\left(A\left(|Du|\right)D_j u\right)\right)\right)\\
	 &=C_{n-1}^{k-1}\left(\frac{A\left(\varphi^{\prime}(r)\right)\varphi^{\prime}(r)}{r}\right)^{k-1}\left(\left(A\left(\varphi^{\prime}(r)\right)\varphi^{\prime}(r)\right)'+\frac{n-k}{k} \frac{A\left(\varphi^{\prime}(r)\right)\varphi^{\prime}(r)}{r}\right) \\
	& =f^{k}(\varphi(r))>0,
	\end{aligned}
	$$
	which means
	$$
	\left(A\left(\varphi^{\prime}(r)\right)\varphi^{\prime}(r)\right)'+\frac{n-k}{k} \frac{A\left(\varphi^{\prime}(r)\right)\varphi^{\prime}(r)}{r}>0,
	$$
	and then for $1 \leq l \leq k$,
	$$
	\begin{aligned}
	\sigma_{l}&\left(\lambda\left(D_i\left(A\left(|Du|\right)D_j u\right)\right)\right)\\ &=C_{n-1}^{l-1}\left(\frac{A\left(\varphi^{\prime}(r)\right)\varphi^{\prime}(r)}{r}\right)^{l-1}\left(\left(A\left(\varphi^{\prime}(r)\right)\varphi^{\prime}(r)\right)'+\frac{n-l}{l} \frac{A\left(\varphi^{\prime}(r)\right)\varphi^{\prime}(r)}{r}\right) \\
	& \geq C_{n-1}^{l-1}\left(\frac{A\left(\varphi^{\prime}(r)\right)\varphi^{\prime}(r)}{r}\right)^{l-1}\left(\left(A\left(\varphi^{\prime}(r)\right)\varphi^{\prime}(r)\right)'+\frac{n-k}{k} \frac{A\left(\varphi^{\prime}(r)\right)\varphi^{\prime}(r)}{r}\right)  \\
	&>0.
	\end{aligned}
	$$
	This implies that $\lambda\left(D_i\left(A\left(|Du|\right)D_j u\right)\right)\in \Gamma_{k}$  is valid in $B_R$.
\end{proof}
	Obviously for $u(x)=\varphi(r)$,  we can see that  $u(x) \in C^{2}\left(B_{R} \backslash\{0\}\right) \cap C^{1}\left(B_{R}\right)$, with $A\left(|Du|\right)D u \in C^{1}\left(B_{R}\right) $	is a solution of  (\ref{221}) if and only if $\varphi(r)\in C^{1}[0, R) \cap C^{2}(0, R) $  with $A\left(\varphi^{\prime}\left(r\right)\right)\varphi^{\prime}\left(r\right)\in C^{1}[0, R)$ is a solution of (\ref{203}).
\begin{remark}\label{212}
	In particular, if $A\in C^1[0,\infty),\ A(0)\neq 0
	$, by Remark \ref{214}, we have $\varphi\left(r\right)\in C^{2}[0, R),$ and
$$
\lim _{x \rightarrow 0} u_{ij}(x)=\lim _{x \rightarrow 0}\left(\left(\varphi^{\prime \prime}(r)-\frac{\varphi^{\prime}(r)}{r}\right) \frac{x_{i} x_{j}}{r^{2}}+\left(\frac{\varphi^{\prime}(r)}{r}\right) \delta_{i j}\right)
=\varphi^{\prime \prime}(0) \delta_{i j}.
$$
Difine $u_{ij}(0)=\varphi^{\prime \prime}(0) \delta_{i j}$,
then it is straightforward to show that $u(x) \in C^{2}\left(R_{R}\right)$.
Otherwise we assume (\ref{211}),
then by $(\ref{213})$ and $\left|\frac{x_{i} x_{j}}{r^{2}}\right| \leq 1$, we have
$$
\begin{aligned}
\limsup _{r \rightarrow 0} \frac{|u_{ij}(x)|}{r^{-\frac{l-2}{l-1}}}
 &=\limsup _{r \rightarrow 0} \frac{\left|\left(\varphi^{\prime \prime}(r)-\frac{\varphi^{\prime}(r)}{r}\right) \frac{x_{i} x_{j}}{r^{2}}+\left(\frac{\varphi^{\prime}(r)}{r}\right) \delta_{i j}\right|}{r^{-\frac{l-2}{l-1}}} \\
& \leq \lim _{r \rightarrow 0} \frac{\left|\varphi^{\prime \prime}(r)\right|+\left|\frac{\varphi^{\prime}(r)}{r}\right|+\left|\frac{\varphi^{\prime}(r)}{r}\right|}{r^{-\frac{l-2}{l-1}}} \\
&=3 \left(\frac{f^{k}(a)}{C_n^k}\right)^{\frac{1}{k}}.
\end{aligned}
$$
For $\frac{l-2}{l-1}q<1$, we have $D^{2} u(x) \in L^{n q}\left(B_{R}\right)$. Then it is straightforward to see $u(x) \in W^{2, n q}\left(B_{R}\right)$.
\end{remark}
Next, we dicuss the local existence of $(\ref{204})$ near $r=0$. The equipment we use is Euler's break line, and the process is similar to the proof of the existence theorem of  ordinary differential equations  (see \cite{r17}).		
\begin{lemma}\label{219}
	For any positive number a, there exist a positive number $R$ such that the Cauchy problem (\ref{204}) has a solution in $[0, R]$.
\end{lemma}	
\begin{proof}
	By  Lemma \ref{206}, we know that $\varphi'(r)=g^{-1}\left(F(r,\varphi)\right)\in C[0, R) \cap C^{1}(0, R)$ is a strictly monotone increasing  function of $r$.
	We define a functional $G(\cdot, \cdot)$ on
	$$
	\mathcal{R}:=[0, R] \times\left\{\varphi \in C[0, R]: a \leq \varphi<2 a\right\},
	$$
	as
	$$
	G(r, \varphi):=g^{-1}\left(F(r,\varphi)\right),
	$$
	where $R$ is a small enough positive constant. Then (\ref{204}) can be rewritten as
	$$
	\varphi^{\prime}(r)=G(r, \varphi).
	$$
	It is easy to see $G>0$ for $r>0$.
	
	We defined a Euler's break line on $[0, R]$ as
	\begin{equation}\label{215}
	\left\{\begin{array}{l}
	\psi(r)=a, \ 0 \leq r \leq r_{1}, \\
	\psi(r)=\psi\left(r_{i-1}\right)+G\left(r_{i-1}, \psi\right)\left(r-r_{i-1}\right), \ r_{i-1}<r \leq r_{i}, \ i=2,3, \cdots, m,
	\end{array}\right.
	\end{equation}
	where $0=r_{0}<r_{1}<\cdots<r_{m}=R$ and $m \in \mathbb{N}$.
	
	\textbf{Step 1.} We want to make sure that $a \leq \psi(r)<2 a$ for all $r \in[0, R]$, \ i.e.   $(r,\psi)\in \mathcal{R}$.
	In fact, it is obvious $\psi(r) \geq a$. Since
	\begin{equation}\label{216}
	\begin{aligned}
	G\left(r, \psi\right)
		&\leq g^{-1}\left(\left(\frac{nr^{k-n}}{C_{n}^k} \int_{0}^{r} s^{n-1}  d s f^{k}(\psi(r))\right)^{\frac{1}{k}}\right)\\
	&\leq g^{-1}\left(\left(\frac{1}{C_{n}^k} \right)^{\frac{1}{k}}  R f(\psi(R))\right)<\infty.
	\end{aligned}
	\end{equation}	
	Then for the break line $(r, \psi)$, we have
	$$
	a \leq \psi(r) \leq a+g^{-1}\left(\left(\frac{1}{C_{n}^k} \right)^{\frac{1}{k}}  R f(\psi(R))\right) r \leq a+g^{-1}\left(\left(\frac{1}{C_{n}^k} \right)^{\frac{1}{k}}  R f(\psi(R))\right)R.
	$$	
	Therefore, we can choose $R$ sufficiently small to make sure that $\psi(r)<2 a$.
	
	\textbf{Step 2.} We will prove that Euler's break line $\psi$ is an $\varepsilon$-appromation solution of (\ref{204}). To do this, we only need to prove that for any small $\varepsilon>0$, there are appropriate points $\left\{r_{i}\right\}_{i=1, \cdots, m}$ to make the break line satisfy
	\begin{equation}
	\left|\frac{d \psi(r)}{d r}-G(r, \psi)\right|<\varepsilon, \ r \in[0, R].
	\end{equation}
	By (\ref{215}), $\psi(r)$ has continuous derivatives in $[0, R]$ expect for a few points. There are unilateral derivatives at these individual points. If the derivative doesn't exist, we consider the right derivative.
	
	As a matter of fact, by (\ref{216}), it is easy to see that
	$$
	\lim _{r \rightarrow 0} G(r, \psi)=0
	$$
	is valid uniformly for any $(r, \psi) \in \mathcal{R}$. Then for each $\varepsilon>0$, there exists $\bar{r}\in(0,R)$ such that for $0 \leq r<\bar{r}$, we have
	$$
	G(r, \psi)<\varepsilon.
	$$	
	Assume $r_{1}=\bar{r}$, then
	$$
	\left|\frac{d \psi(r)}{d r}-	G(r, \psi)\right|=|	G(r, \psi)|<\varepsilon,\ 0<r<\bar{r}.
	$$	
	For $\bar{r} \leq r \leq R$,  by the proof of Lemma \ref{206}, we know that $g^{-1}\in C[0,F(R,\psi)]\cap C^1(0,F(R,\psi)]$, then $g^{-1}$ is Liptchitz continuous on $[F(\bar{r},\psi),F(R,\psi)]$. Let $r_{i-1}<r \leq r_{i}$, we have
	\begin{equation*}
	\begin{aligned}
	&\left|\frac{d \psi(r)}{d r}-G(r, \psi)\right|
\leq C\left|F(r_{i-1},\psi)-F(r,\psi)\right|\\
	\leq& C\left(\frac{n}{C_n^k}\right)^{\frac{1}{k}} \left(
	r^{k-n} \int_{0}^{r} s^{n-1} f^{k}(\psi(s)) d s-r_{i-1}^{k-n} \int_{0}^{r_{i-1}} s^{n-1} f^{k}(\psi(s)) d s \right)^{\frac{1}{k}}\\
	\leq& C\left(\frac{n}{C_n^k}\right)^{\frac{1}{k}} \left(\left(r_{i-1}^{k-n}-r^{k-n}\right) \int_{0}^{r} s^{n-1} f^{k}(\psi(s)) d s
	+r_{i-1}^{k-n} \int_{r_{i-1}}^{r} s^{n-1} f^{k}(\psi(s)) d s \right)^{\frac{1}{k}}\\
	\leq &C\left(\frac{1}{C_n^k}\right)^{\frac{1}{k}} \left(\left(r_{i-1}^{k-n}-r^{k-n}\right) R^n f^{k}(2a) 
	+\bar{r}^{k-n} \left(r^{n}-r_{i-1}^{n} \right) f^{k}(2a)  \right)^{\frac{1}{k}}.
	\end{aligned}
	\end{equation*}
	 Since function $r^{k-n}$ and $r^{n}$ are both  Liptchitz continuous on $[\bar{r}, R]$, for the above $\varepsilon$, there exists $\delta\left(\varepsilon\right)>0$ satisfying  
	$$
	\max _{2 \leq i \leq m}\left|r_{i-1}-r_{i}\right|<\delta\left(\varepsilon\right),
	$$
	and then we have 
	$$\left|\frac{d \psi(r)}{d r}-G(r, \psi)\right|<C|r_{i-1}-r|<\varepsilon.$$
Thus, Euler's break line $\psi$ is an $\varepsilon$-appromation solution of (\ref{204}).
	
	\textbf{Step 3.} The next step is to find a solution of (\ref{204}) by the Euler break line we defined. Assume $\left\{\varepsilon_{j}\right\}_{j=1}^{\infty}$ is a positive constant sequence converging to $0$. For each $\varepsilon_{j}$, there is an $\varepsilon_{j}$-appromation solution $\psi_{j}$ on $[0, R]$, defined as above. By Step 1, it is easy to know that
	$$
	\left|\psi_{j}\left(r^{\prime}\right)-\psi_{j}\left(r^{\prime \prime}\right)\right|=G(r_{i-1},\psi_j)|r'-r''| \leq M\left|r^{\prime}-r^{\prime \prime}\right|,
	$$
	where $(r^{\prime},\psi_{j}), (r^{\prime \prime},\psi_{j}) \in\mathcal{R}$. That is to say, $\left\{\psi_{j}\right\}$ is equicontinuous and uniformaly bounded $(r''=0)$. Then by the Ascoli-Arzela Lemma, we can find a uniformly convergent subsequence, still denoted as $\left\{\psi_{j}\right\}$, without loss of generality.
	
	Assume $\lim _{j \rightarrow \infty} \psi_{j}=\varphi$. Since $\psi_{j}\in C[0, R]$, we know that $\varphi\in C[0, R]$. By $\psi_{j}(0)=a$, we have $\varphi(0)=a$.
	
	Since $\psi_{j}$ is an $\varepsilon_{j}$-appromation solution, we have
	\begin{equation}\label{217}
	\frac{d \psi_{j}(r)}{d r}=G(r,\psi_j)+\Delta_{j}(r),
	\end{equation}
	where $\left|\Delta_{j}(r)\right|<\varepsilon_{j}$, for $r \in[0, R]$. Intergrating (\ref{217}) from 0 to $r(\leq R)$, we have
	$$
	\psi_{j}(r)=a+\left(\int_{0}^{r} G(s,\psi_j) d s+\int_{0}^{r} \Delta_{j}(s) d s\right).
	$$
	Let $j \rightarrow \infty$,
	\begin{equation}\label{218}
	\begin{aligned}
	\varphi(r) &=a+\lim _{j \rightarrow \infty}\left(\int_{0}^{r} G(s,\psi_j) d s+\int_{0}^{r} \Delta_{j}(s) d s\right) \\
	&=a+\int_{0}^{r} G(s,\varphi) d s .
	\end{aligned}
	\end{equation}
	Since $\varphi\in C[0, R]$, by $(\ref{218})$,   $\varphi\in C^1(0, R]$. Differentiating (\ref{218}), we have
	$
	\varphi^{\prime}(r)=G(r,\varphi),\ r>0.
	$
	Hence, we can see that $\varphi$ satisfies $(\ref{204})$ in $[0, R]$.	
\end{proof}		
In fact, a local solution also exists for any real number $a$ if we do not consider only the positive ones. Once $a$ is positive, it is easy to know the solution $\varphi$ is positive, too.

\section{Proof of the Main Results}

	We will prove the main results by the comparison lemma.
	\begin{lemma}\label{305}
	 Let $\varphi(r) \in  C^{1}[0, R) \cap C^{2}(0, R)$ with $A\left(\varphi^{\prime}\left(r\right)\right)\varphi^{\prime}\left(r\right)\in C^{1}[0, R)$ satisfying (\ref{203}), with $\varphi^{\prime}(0)=0$ and $\varphi(r) \rightarrow \infty$ as $r \rightarrow R$. Then, if $u(x) \in  C^{2}\left(\mathbb{R}^{n} \backslash\{0\}\right) \cap \Phi^{ k}\left(\mathbb{R}^{n}\right)$ is a positive solution of $(\ref{101})$, we have $u(x) \leq \varphi(|x|)$ at each point in $B_{R}$.
	\end{lemma}

	\begin{proof}
		Let $v(x)=\varphi(|x|)$, and  by Lemma $\ref{205}$,  we know    $v(x) \in C^{2}\left(B_{R} \backslash\{0\}\right) \cap C^{1}\left(B_{R}\right)$	is a solution of  (\ref{221}).

		Let $L[w]=\sigma_{k}^{\frac{1}{k}}\left(\lambda\left(D_i\left(A\left(|Dw|\right)D_j w\right)\right)\right)-f(w)$. Suppose to the contrary that $u>v$ somewhere, then there is some constant $a>0$ such that $u-a$ touches $v$ from below, which means $u-a-v \leq 0$ in $B_{R}$. Suppose $u-a$ touches $v$ at some interior point $x_{0}$ in $B_{R}$. Then there is $R^{\prime} \in(0, R)$ such that $x_{0} \in B_{R^{\prime}}$. Since $v(x)=\varphi(|x|) \rightarrow \infty$ as $x \rightarrow \partial B_{R}$ and $u$ is bounded in $B_{R}$, we can assume $\sup _{\partial B_{R^{\prime}}}(u-a-v)<0$.
		
		It follows from (\ref{109}) that in $B_{R}^{\prime}$,
		$$
		\begin{aligned}
		L[u-a] &=\sigma_{k}^{\frac{1}{k}}\left(\lambda\left(D_i\left(A\left(|D\left(u-a\right)|\right)D_j \left(u-a\right)\right)\right)\right)
		-f(u-a) \\
		&=\left(\sigma_{k}^{\frac{1}{k}}\left(\lambda\left(D_i\left(A\left(|Du|\right)D_j u\right)\right)\right)-f(u)\right)+(f(u)-f(u-a)) \\
		& \geq 0=L[v].
		\end{aligned}
		$$
		Now $u-a$ is a subsolution and $v$ is a solution (with respect to $\mathrm{L}$ ). By the maximum principle,
		$$
		0=\sup _{B_{R^{\prime}}}(u-a-v)=\sup _{\partial B_{R^{\prime}}}(u-a-v)<0,
		$$
		which is impossible.
	\end{proof}
	
			\begin{lemma}\label{306}
		The inequality (\ref{101}) has a positive  solution $u \in C^{2}\left(\mathbb{R}^{n} \backslash\{0\}\right) \cap \Phi^{k}\left(\mathbb{R}^{n}\right)$ if and only if  the Cauchy problem (\ref{204}) has a positive solution $\varphi(r) \in C^{2}(0, \infty) \cap C^{1}[0,\infty)$  with $A\left(\varphi^{\prime}\left(r\right)\right)\varphi^{\prime}\left(r\right)\in C^{1}[0,\infty)$ for some positive number $a .$
	\end{lemma}
\begin{proof}
	First, the sufficient condition is obvious. If there exists such a solution $\varphi(r)$ of $(\ref{204})$ for $R=+\infty$, let $u(x)=\varphi(r),\ r=|x|$. By Lemma \ref{205} and Lemma \ref{206}, 
	$\sigma_{k}^{\frac{1}{k}}\left(\lambda\left(D_i\left(A\left(|Du|\right)D_j u\right)\right)\right) = f\left(u\right)$ and $\lambda\left(D_i\left(A\left(|Du|\right)D_j u\right)\right) \in \Gamma_{k}$ for $x \in \mathbb{R}^{n}$. Thus $u(x) \in C^{2}\left(\mathbb{R}^{n} \backslash\{0\}\right) \cap \Phi^{k}\left(\mathbb{R}^{n}\right)$ is a required solution of (\ref{101}).

	Next, we will prove the necessary condition. On the contrary, suppose that no such function $\varphi(r)$ exists globally. Then for any positive number $a$, the Cauchy problem (\ref{204}) has a positive solution $\varphi(r)$ on some interval which cannot be a global solution. Hence, we assume $[0, R)$ is the maximal interval in which the solution exists. Since $\varphi^{\prime}(r)>0$ for $r>0$, we know $\varphi(r) \rightarrow \infty$ as $r \rightarrow R$. Then  by Lemma \ref{206}, $\varphi(|x|)$ satisfies (\ref{203}). By Lemma \ref{305},  any positive solution $u(x) \in C^{2}\left(\mathbb{R}^{n} \backslash\{0\}\right) \cap \Phi^{k}\left(\mathbb{R}^{n}\right)$ of $(\ref{101})$ would satisfy $u(x) \leq \varphi(|x|)$ for $x \in B_{R}$. In particular we have $u(0) \leq \varphi(0)=a$. However, since $a$ is arbitrary, we take $a=\frac{u(0)}{2}$ and obtain a contradiction, which means the necessary condition holds.
\end{proof}
\begin{proof}[Proof of Theorem {\upshape\ref{102}}]
Suppose to the contrary that inequality (\ref{101}) has a positive  solution $u \in C^{2}\left(\mathbb{R}^{n} \backslash\{0\}\right) \cap \Phi^{k}\left(\mathbb{R}^{n}\right)$. Then, by Lemma \ref{306}, the Cauchy problem (\ref{204}) has a positive solution $\varphi(r)\in C^{2}(0, \infty)\cap   C^{1}[0, \infty) $  with $A\left(\varphi^{\prime}\left(r\right)\right)\varphi^{\prime}\left(r\right)\in C^{1}[0,\infty)$ for some positive number a.
Since $f$ and $\varphi$ are both monotone non-decreasing, it follows from $(\ref{204})$ that
\begin{equation}\label{401}
A\left(\varphi^{\prime}(r)\right) \varphi^{\prime}(r)
=\left(\frac{n r^{k-n}}{C_n^k} \int_{0}^{r} s^{n-1} f^k(\varphi(s)) d s\right)^{\frac{1}{k}}\leq \left(\frac{1}{C_n^k}\right)^{\frac{1}{k}}rf\left(\varphi(r)\right)
, \ r>0.
\end{equation}
Substituting $(\ref{401})$ in (\ref{203}), we obtain
$$C_{n-1}^{k-1}\left(A\left(\varphi^{\prime}(r)\right)\varphi^{\prime}(r)\right)'\left(\frac{A\left(\varphi^{\prime}(r)\right)\varphi^{\prime}(r)}{r}\right)^{k-1}
\geq\frac{k}{n}f^k\left(\varphi\left(r\right)\right),\ r>0,$$
which comes to
\begin{equation}\label{402}
\left(\left(A\left(\varphi^{\prime}(r)\right)\varphi^{\prime}(r)\right)^k\right)'
\geq
\frac{k}{C_n^k}r^{k-1}f^k\left(\varphi\left(r\right)\right),\ r>0.
\end{equation}
We now integrate $(\ref{402})$ over $[0, r]$ to obtain
\begin{equation*}
\left(A\left(\varphi^{\prime}(r)\right)\varphi^{\prime}(r)\right)^k
\geq\frac{k}{C_n^k}\int_{0}^{r}s^{k-1}f^k\left(\varphi\left(s\right)\right) ds
\geq\frac{1}{C_n^k}r^{k}f^k\left(a\right),\ r>0,
\end{equation*}
which leads to
\begin{equation}\label{403}
A\left(\varphi^{\prime}(r)\right)\varphi^{\prime}(r)
\geq
\left(\frac{1}{C_n^k}\right)^{\frac{1}{k}} rf\left(a\right),\ r>0.
\end{equation}
By (\ref{103}), we see that
\begin{equation*}
\left(\frac{1}{C_n^k}\right)^{\frac{1}{k}} rf\left(a\right)
\leq A\left(\varphi^{\prime}(r)\right)\varphi^{\prime}(r) \leq 
\lim _{p \rightarrow \infty} p A(p)<\infty, \ r>0.
\end{equation*}
Letting $r \rightarrow \infty$ in the above, we have a contradiction. This completes the proof.
\end{proof}
	Next we consider some properties of the function (\ref{107}). By (\ref{103}), we know
$$
\Psi'(p)=p \left(\left(p A\left(p\right)\right)^k\right)'> 0, \ p > 0,
$$
then $\Psi$ is strictly monotone increasing on $(0,\infty)$ and $\Psi(0)=0$. By 
$$
\Psi(p)+\int_{0}^{1} \left(tA(t)\right)^k d t
=p \left(p A\left(p\right)\right)^k-\int_{1}^{p} \left(t A(t)\right)^k d t > \left(p A\left(p\right)\right)^k, \ p>1,
$$
we have $\lim _{p \rightarrow \infty} \Psi(p)=\infty$.
Hence the inverse function of $\Psi$ on $[0,\infty)$ exists, denoted by $\Psi^{-1}$. Clearly $\Psi^{-1}$ is a strictly monotone increasing function and satisfies $\lim _{p \rightarrow \infty} \Psi^{-1}(p)=\infty$.

 \begin{lemma}\label{104}
		Assume that $A$  satisfies (\ref{103}), (\ref{116}) and $f$ satisfies (\ref{109}). If
	\begin{equation}\label{106}
	\int^{\infty}\left(\Psi^{-1}\left(\int^{s}f^k\left(t\right) dt\right)\right)^{-1} d s<\infty,
	\end{equation}
	then the inequality (\ref{101}) has no positive solution $u\in C^{2}\left(\mathbb{R}^{n} \backslash\{0\}\right) \cap \Phi^{k}\left(\mathbb{R}^{n}\right)$ .
\end{lemma}

\begin{proof}
 Suppose to the contrary that the inequality (\ref{101}) has a positive  solution $u \in C^{2}\left(\mathbb{R}^{n} \backslash\{0\}\right) \cap \Phi^{k}\left(\mathbb{R}^{n}\right)$. Then, by Lemma \ref{306}, the Cauchy problem (\ref{204}) has a positive solution $\varphi(r)\in C^{2}(0, \infty) \cap C^{1}[0, \infty)$  with $A\left(\varphi^{\prime}\left(r\right)\right)\varphi^{\prime}\left(r\right)\in C^{1}[0,\infty)$ for some positive number a. 
  As above we obtain (\ref{402}) and (\ref{403}). Letting $r \rightarrow \infty$ in (\ref{403}), by (\ref{103}), we have $\lim _{r \rightarrow \infty} \varphi^{\prime}(r)=\infty$. Therefore $\lim _{r \rightarrow \infty} \varphi(r)=\infty$. Multiplying (\ref{402}) by $\varphi^{\prime}>0$ , we have
  $$
 \Psi'\left(\varphi^{\prime}(r)\right)
 =\varphi^{\prime}(r)\left(\left(A\left(\varphi^{\prime}(r)\right)\varphi^{\prime}(r)\right)^k\right)'
 \geq\frac{k}{C_n^k}f^k\left(\varphi\left(r\right)\right)\varphi^{\prime}(r),\ r> 1,
 $$
 and then integrating on $[1, r]$, we obtain
   $$
\Psi\left(\varphi^{\prime}(r)\right)
\geq\frac{k}{C_n^k}\int_{\varphi\left(1\right)}^{\varphi\left(r\right)}f^k\left(s\right) ds, \ r>1.
$$
Hence
 $$
\left(\Psi^{-1}\left(\frac{k}{C_n^k}\int_{\varphi\left(1\right)}^{\varphi\left(r\right)}f^k\left(s\right) ds\right)\right)^{-1} \varphi^{\prime}(r) \geq 1,\ r> 1.
$$
Integrating on $[1, r]$, we have
\begin{equation}\label{404}
\int_{\varphi(1)}^{\varphi(r)}\left(\Psi^{-1}\left(\frac{k}{C_n^k}\int_{\varphi\left(1\right)}^{s}f^k\left(t\right) dt\right)\right)^{-1} d s \geq r-1, \ r>1.
\end{equation}
Letting $r \rightarrow \infty$ in (\ref{404}), we have
\begin{equation*}
\int_{\varphi(1)}^{\infty}\left(\Psi^{-1}\left(\frac{k}{C_n^k}\int_{\varphi\left(1\right)}^{s}f^k\left(t\right) dt\right)\right)^{-1} d s = \infty,
\end{equation*}
which contradicts (\ref{106}). This completes the proof.
\end{proof}

 \begin{lemma}\label{105}
	Assume that $A$  satisfies (\ref{103}), (\ref{116}) and $f$ satisfies (\ref{109}).  If (\ref{110}) holds,
	then the inequality (\ref{101}) has a positive  solution $u\in C^{2}\left(\mathbb{R}^{n} \backslash\{0\}\right) \cap \Phi^{k}\left(\mathbb{R}^{n}\right)$.
\end{lemma}

\begin{proof}
By Lemma \ref{306}, we only need to prove that the 
Cauchy problem (\ref{204}) has a positive solution 
$\varphi(r)\in C^{2}(0, \infty)\cap C^{1}[0, \infty)$  with $A\left(\varphi^{\prime}\left(r\right)\right)\varphi^{\prime}\left(r\right)\in C^{1}[0,\infty)$ for some positive number $a$.
 Suppose to the contrary that  no such solution of (\ref{204}) exists. As in the proof of Lemma \ref{306}, the problem (\ref{204}) has such a positive solution $\varphi(r)$  valid on the maximal existence interval $[0, R)$. By Lemma $\ref{206}$, we know that $\varphi$ satisfies  (\ref{203}).
 
 Next, we show that $\varphi(R)=\lim_{r \rightarrow R}\varphi(r)=\infty,\ r\in [0, R)$. Suppose to the contrary that $\varphi(R)<\infty$. Then, by (\ref{204}),  $\varphi^{\prime}(R)<\infty$  exists. By the continuation theorem of the Cauchy problem (\ref{204}), $\varphi$ can be extended as a solution to the right beyond $R$. This contradicts the choice of $R$. Hence we  have $\varphi(R)=\infty$.
 
Since $\varphi^{\prime}(r) > 0$ for $0 < r<R$, then by (\ref{203}) and (\ref{103}), we have
		\begin{equation*}
		C_{n-1}^{k-1} \left(A\left(\varphi^{\prime}(r)\right)\varphi^{\prime}(r)\right)'\left(\frac{A\left(\varphi^{\prime}(r)\right)\varphi^{\prime}(r)}{r}\right)^{k-1}
\leq
f^k\left(\varphi\left(r\right)\right),\ 0<r<R.
\end{equation*}	
which comes to
\begin{equation*}
\left(\left(A\left(\varphi^{\prime}(r)\right)\varphi^{\prime}(r)\right)^k\right)'
\leq \frac{n}{C_n^k}r^{k-1} f^k\left(\varphi\left(r\right)\right), \ 0<r<R.
\end{equation*}
Multiplying the above by $\varphi^{\prime} > 0$, we have
 $$
\Psi'\left(\varphi^{\prime}(r)\right)
=\varphi^{\prime}(r)\left(\left(A\left(\varphi^{\prime}(r)\right)\varphi^{\prime}(r)\right)^k\right)'
\leq\frac{n R^{k-1}}{C_n^k}f^k\left(\varphi\left(r\right)\right)\varphi^{\prime}(r), \ 0<r<R,
$$
and then integrating on $[0, r],\ r<R$, we obtain
	   $$
\Psi\left(\varphi^{\prime}(r)\right)
\leq\frac{n R^{k-1}}{C_n^k}\int_{a}^{\varphi\left(r\right)}f^k\left(s\right) ds, \ 0<r<R.
$$
Hence
	$$
\left(\Psi^{-1}\left(\frac{n R^{k-1}}{C_n^k}\int_{a}^{\varphi\left(r\right)}f^k\left(s\right) ds\right)\right)^{-1} \varphi^{\prime}(r) \leq 1, \ 0<r<R.
$$
Integrating on $[0, r], \ r<R$, we have
	\begin{equation*}
\int_{a}^{\varphi(r)}\left(\Psi^{-1}\left(\frac{n R^{k-1}}{C_n^k}\int_{a}^{\varphi\left(r\right)}f^k\left(s\right) ds\right)\right)^{-1} d s 
\leq r, \ 0<r<R.
\end{equation*}
Letting $r \rightarrow R$ in the above, we have
		\begin{equation*}
\int_{a}^{\infty}\left(\Psi^{-1}\left(\frac{n R^{k-1}}{C_n^k}\int_{a}^{s}f^k\left(t\right) dt\right)\right)^{-1} d s 
\leq R<\infty,
\end{equation*}	
which contradicts (\ref{110}).  This  completes the proof.
\end{proof}
Combining Lemma \ref{104} and Lemma \ref{105}, we proof Theorem \ref{115}.

\begin{lemma}\label{307}
	Assume that (\ref{108}) holds. Then
	$$
	0<\liminf_{p \rightarrow \infty}  \frac{\Psi(p)}{p^{k(m-1)+1}} \leq \limsup_{p \rightarrow \infty} \frac{\Psi(p)}{p^{k(m-1)+1}}<\infty.
	$$
	Consequently, we have
	$$
	0<\liminf_{p \rightarrow \infty}  \frac{\Psi^{-1}(p)}{p^{\frac{1}{k(m-1)+1}}} \leq \limsup_{p \rightarrow \infty} \frac{\Psi^{-1}(p)}{p^{\frac{1}{k(m-1)+1}}}<\infty.
	$$
\end{lemma}
\begin{proof}
	By (\ref{107}) and (\ref{108}), we have
	$$\limsup_{p \rightarrow \infty} \frac{\Psi(p)}{p^{k(m-1)+1}}
	\leq
	\limsup _{p \rightarrow \infty}
	\frac{p \left(p A\left(p\right)\right)^k}{p^{k(m-1)+1}}
	<\infty.$$	
	Next we show that there exist positive constants $P$ and $C$ such that	
	\begin{equation}\label{301}
	\Psi(p) \geq C p^{k(m-1)+1}, \ p \geq P,
	\end{equation}
	which implies that 
	$$\liminf_{p \rightarrow \infty}  \frac{\Psi(p)}{p^{k(m-1)+1}}>0.$$
	The condition (\ref{108}) implies the existence of positive constants $P_{1}, C_{1}$ and $C_{2}$ such that
	\begin{equation}\label{302}
	C_{1} p^{k(m-1)+1} 
	\leq
	p \left(p A\left(p\right)\right)^k 
	\leq 
	C_{2} p^{k(m-1)+1},\ p \geq P_{1} .
	\end{equation}
	By (\ref{302}), choose $\theta>0$ so small that $\left(C_{2} / C_{1}\right) \theta^{k(m-1)}<1 / 2$, then we have	
	\begin{equation}\label{303}
	\frac{\left(\theta p\right)\left(\left(\theta p\right) A\left(\theta p\right)\right)^k}{p \left(p A\left(p\right)\right)^k} 
	\leq 
	\frac{C_{2} \left(\theta p\right)^{k(m-1)+1}}{C_{1} p^{k(m-1)+1}}
	=
	\frac{C_{2} \theta^{k(m-1)+1}}{C_{1}}
	<\frac{1}{2} \theta,\ p \geq P,
	\end{equation}
	where $P=P_{1} / \theta$. We observe that
	\begin{equation}\label{304}
	\begin{aligned}
	\int_{0}^{p} \left(t A(t)\right)^k d t &=\int_{0}^{\theta p} \left(t A(t)\right)^k d t
	+\int_{\theta p}^{p} \left(t A(t)\right)^k d t \\
	& \leq \left(\theta p\right)\left(\left(\theta p\right) A\left(\theta p\right)\right)^k
	+ \left(p-\theta p\right)\left(p A\left(p\right)\right)^k\\
	&=p \left(p A\left(p\right)\right)^k
	\left(1-\theta+\frac{\left(\theta p\right) \left(\left(\theta p\right) A\left(\theta p\right)\right)^k}{p \left(p A\left(p\right)\right)^k}\right),\ p \geq P.
	\end{aligned}
	\end{equation}	
	From (\ref{304}), (\ref{303}), and (\ref{302}) it follows that
	\begin{equation*}
	\begin{aligned}
	\Psi(p)
	&=p \left(p A\left(p\right)\right)^k\left(1-\frac{\int_{0}^{p} \left(t A(t)\right)^k d t}{p \left(p A\left(p\right)\right)^k}\right) \\
	& \geq
	p \left(p A\left(p\right)\right)^k
	\left(\theta-\frac{\left(\theta p\right) \left(\left(\theta p\right) A\left(\theta p\right)\right)^k}{p \left(p A\left(p\right)\right)^k}\right) \\
	&>\frac{1}{2} \theta p \left(p A\left(p\right)\right)^k\\
	&\geq \frac{1}{2} \theta C_{1} p^{k(m-1)+1},\ p \geq P,
	\end{aligned}
	\end{equation*}	
	which gives (\ref{301}). This completes the proof.
\end{proof}
By Theorem \ref{115} and Lemma  \ref{307}, we can get  Corollary \ref{113} immediately.
\begin{proof}[Proof of Corollary {\upshape\ref{118}}]
 The proof of Corollary $\ref{118}$ is  similar to the above. Most of the properties we need are almost the same as we have discussed. Since $f$ is now a positive, continuous and monotone non-decreasing function defined on $\mathbb{R}$, we do not need $a$ to be positive in Lemma \ref{206}, Lemma $\ref{219}$ and Lemma $\ref{306}$ to get the similar conclusions.
 \end{proof}

\end{document}